\documentclass[12pt,a4paper, reqno]{amsart}
\usepackage{amsfonts,amsmath,amssymb}
\usepackage{hyperref}
\usepackage{latexsym,fullpage,amsfonts,amssymb,amsmath,amscd,graphics,epic,enumerate}
\usepackage[all]{xy}
\usepackage{amssymb,amsthm,amsxtra,comment}
\usepackage{color}
\usepackage[dvipsnames]{xcolor}
\usepackage{amscd}
\usepackage{amsthm}
\usepackage{amsfonts,cancel}
\usepackage{amssymb}
\usepackage{url}
\usepackage{bbm}
\usepackage{wasysym}
\usepackage{MnSymbol}
\usepackage{kbordermatrix}
\usepackage{tikz}

\usepgflibrary{shapes.geometric}

\tikzstyle{V}=[draw, fill =black, circle, inner sep=0pt, minimum size=1.5pt]
\tikzstyle{C}=[draw, fill =white, circle, inner sep=0pt, minimum size=1.5pt]
\tikzstyle{over}=[draw=white,double=black,line width=2pt, double distance=.5pt]

\theoremstyle{plain}
\newtheorem*{theorem*}{Theorem}
\newtheorem*{remark*}{Remark}
\newtheorem*{example*}{Example}
\newtheorem{lemma}{Lemma}[subsection]
\newtheorem{proposition}[lemma]{Proposition}
\newtheorem{corollary}[lemma]{Corollary}
\newtheorem{theorem}[lemma]{Theorem}

\newtheorem*{conjecture*}{Conjecture}

\theoremstyle{definition}

\theoremstyle{remark}
\newtheorem{remark}[lemma]{Remark}

 \newcommand{\op}{\operatorname}

\newcommand{\Hom}{\operatorname{Hom}}

\newcommand{\ch}{\operatorname{ch}}
\newcommand{\sch}{\operatorname{sch}}

\newcommand{\Ind}{\operatorname{Ind}}
\newcommand{\Coind}{\operatorname{Coind}}

\newcommand{\sdim}{\operatorname{sdim}}

\newcommand{\Span}{\operatorname{Span}}

\renewcommand{\Im}{\operatorname{Im}}

\newcommand{\Gr}{\operatorname{Gr}}
\newcommand{\sgn}{\operatorname{sgn}}

\renewcommand{\gg}{\mathfrak{g}}

\renewcommand{\dim}{\mathrm{dim}}

\newcommand{\odd}{\mathrm{odd}}

\newcommand{\finite}{\mathrm{finite}}
\newcommand{\Modulo}[1]{\ (\mathrm{mod}\ #1)}
\newcommand{\End}{\mathrm{End}}

\newif\ifpaper
\papertrue

 \def\<{\langle}
  \def\>{\rangle}

\DeclareMathOperator{\tr}{tr}

\def\quotient#1#2{%
    \raise1ex\hbox{$#1$}\Big/\lower1ex\hbox{$#2$}%
}

\allowdisplaybreaks 
\begin{document}

\title{Grothendieck rings of periplectic Lie superalgebras}

\author{Mee Seong Im 
\and Shifra Reif
\and Vera Serganova} 
\address{Department of Mathematical Sciences, United States Military Academy, West Point, NY 10996 USA}
\email{meeseongim@gmail.com}
\address{Department of Mathematics, Bar-Ilan University, Ramat Gan, Israel}
\email{shifra.reif@biu.ac.il} 
\address{Department of Mathematics, University of California at Berkeley, Berkeley, CA 94720 USA}
\email{serganov@math.berkeley.edu} 
\date{\today}

\begin{abstract}  
We describe explicitly the Grothendieck rings of finite-dimensional representations of the periplectic Lie superalgebras.  
In particular,  the Grothendieck ring of the Lie supergroup $P(n)$  is isomorphic to the ring of symmetric polynomials in $x_1^{\pm 1}, \ldots, x_n^{\pm 1}$  whose evaluation $x_1=x_2^{-1}=t$ is independent of $t$. 
\end{abstract}

\maketitle
\bibliographystyle{amsalpha}  
\setcounter{tocdepth}{3}
\section{Introduction}\label{sec:intro}
The Grothendieck group is a fundamental invariant attached to an abelian category. It is defined to be the free abelian group on the objects of the category modulo the relation 
 $[B]=[A]+[C]$ for every exact sequence $0\rightarrow A \rightarrow B \rightarrow C\rightarrow 0$. 
If the category possesses tensor products of objects, then the Grothendieck group inherits a structure of a ring. A beautiful example is the Grothendieck ring $K[GL(n)]$ of the category of finite-dimensional representations of  the general linear group. It is isomorphic to the ring of symmetric Laurent polynomials
\[
K[GL(n)]\cong \mathbb Z\left[x_1^{\pm 1},\ldots, x_n^{\pm 1} \right]^{S_n}
\]
(see for example, \cite[Sec. 23.24]{MR1153249}).
Moreover, the famous Schur polynomials are images of irreducible representations under this isomorphism.

This description generalizes to all semisimple complex Lie algebras. In this case, the category admits complete reducibility and the characters of irreducible representations are given explicitly by the Weyl character formula.   The Grothendieck ring is then isomorphic to the ring $\mathbb{Z}[P]^W$ of  $W$-invariants in the integral group ring $\mathbb{Z}[P]$, where $P$ is the corresponding weight lattice and $W$ is the Weyl group. The isomorphism is given by the character map. 
%


The analogous theory for Lie superalgebras is more difficult: the category of finite-dimensional representations is not semisimple and a general Weyl character formula is unknown. The class of basic classical Lie superalgebras is better understood, as it carries an invariant bilinear form.

 In 2007, A.N. Sergeev and A.P. Veselov  \cite{MR2776360} described the Grothendieck ring for basic classical Lie superalgebras.  Since modules over Lie superalgebras admit a parity shift functor $\Pi$ which does not change the action of the Lie superalgebra, it is natural to consider one of the two quotients of the ring, either by the relation $[M]=[\Pi M] $ or $[M]=-[\Pi M] $. We refer to these quotients as the ring of characters and the ring of supercharacters,  respectively.  
 
 The theorem of A.N. Sergeev and A.P. Veselov  states that the ring of supercharacters is equal to the subring of $\mathbb{Z}[P]^W$, admitting an extra condition which corresponds to the isotropic roots of the Lie superalgebra. This extra condition can be seen as invariance under an action of the Weyl groupoid. In particular, for the general linear Lie supergroup $GL(m|n)$,  the ring of supercharacters is isomorphic to the ring of supersymmetric Laurent polynomials, namely,
 \[
 \left\{ f\in\mathbb Z\left[x_1^{\pm 1},\ldots, x_m^{\pm 1},y_1^{\pm 1},\ldots, y_n^{\pm 1} \right]^{S_m\times S_n} : f|_{x_1=y_1=t} \mbox{ is independent of }t \right\}.
 \]

The periplectic Lie superalgebra $\mathfrak p(n)$ imposes further difficulties than basic classical Lie superalgebras due to the lack of an invariant bilinear form. It was only recently that its representations were understood and translation functors were computed in \cite{BDEHHILNSS} and \cite{BDEHHILNSS2}.

In this paper, we describe the ring of supercharacters of the periplectic Lie superalgebra. We show that it is isomorphic to the ring of supersymmetric functions with a suitable supersymmetry condition. In particular, for the periplectic Lie supergroup $P(n)$, we get the following theorem:
\begin{theorem}
\label{thm:reduced-GR-periplectic}
The ring of supercharacters of $P(n)$ is isomorphic to 
\[ 
J_n := 
 \{ f\in \mathbb{Z}[x_1^{\pm 1},\ldots, x_n^{\pm 1}]^{S_n}: f|_{x_1=x_2^{-1}=t} \mbox{ is independent of }t  \}.
\] 
\end{theorem}

 The inclusion from left to right for Theorem~\ref{thm:reduced-GR-periplectic} 
 is obtained by restriction to rank-one subalgebras as done in \cite[Prop. 4.3]{MR2776360}.
The other inclusion is much more involved. A key tool is the ring homomorphism $ds_n:J(P(n))\rightarrow J(P(n-2))$ induced from the Duflo--Serganova functor. We use the realization of  $ds_n$ as the evaluation map $f\mapsto f|_{x_n=x_{n-1}^{-1}=t}$, proven in \cite{hoyt2016grothendieck} as well as the description of its kernel. The main step is to prove that $ds_n$ is surjective in order to apply an inductive argument. We construct preimages of $ds_n$ using Euler characteristics of  parabolic inductions given in   \cite{MR2734963} and translation functors given in \cite{BDEHHILNSS}.
 
The description of the ring of supercharacters of the Lie supergroup $SP(n)$ and the Lie superalgebras $\mathfrak p(n)$ and $\mathfrak {sp}(n)$ are deduced from the one of $P(n)$. We also express the character ring of $\mathfrak p(n)$ as the ring of invariant functions under a Weyl groupoid corresponding to the root system of $\mathfrak p (n)$.
 \ifpaper
\subsection*{Structure of the paper} 
In Section~\ref{section:periplectic-salg}, we summarize the representation theory of the periplectic Lie superalgebra $\mathfrak{p}(n)$. In Section~\ref{section:DS-functor}, we define the Duflo--Serganova functor for $P(n)$, 
  construct the corresponding homomorphism between supercharacter rings
  and compute its kernel. The kernel is explicitly described in terms of the supercharacters of thin Kac modules (see Proposition~\ref{lemma:kernel-DS-Pn}).
In Section~\ref{section:main-thm-proof}, we prove the surjectivity of the Duflo--Serganova homomorphism for $P(n)$, 
and in Section~\ref{subsection:main-result-proof}, we prove Theorem~\ref{thm:reduced-GR-periplectic}. 
We then describe the Grothendieck ring of the periplectic Lie superalgebra in Section~\ref{subsubsection:extend-results-to-pn}, and the special periplectic Lie superalgebra in  Section~\ref{subsubsection:extend-results-to-spn}. We end this manuscript by describing the super Weyl groupoid for $\mathfrak{p}(n)$ in  Section~\ref{subsection:super-Weyl-groupoid}.

\subsection*{Acknowledgments}
We thank Bar-Ilan university at Ramat Gan, Israel for hosting M.S.I. and providing excellent collaboration conditions, and  Malka Schaps for interesting discussions.
This project is partially supported by ISF Grant No. $1221/17$, NSF grant $1701532$, and United States Military Academy's Faculty Research Fund.

\section{The periplectic Lie superalgebra and its representations}
\label{section:periplectic-salg}

\subsection{Lie superalgebras}

Given a $\mathbb{Z}_2$-graded vector superspace $V = V_{\bar 0}\oplus V_{\bar 1}$, 
the parity of a homogeneous (even) vector $v\in V_{\bar 0}$ is defined as $\bar v=\bar 0\in \mathbb{Z}_2 =\{ \bar 0 , \bar 1 \}$  while the parity of an odd vector $v\in V_{\bar 1}$ is defined as $\bar v = \bar 1$. If the parity of a vector $v$ is $\bar 0$ or $\bar 1$, we say that $v$ has degree $0$ or $1$, respectively.  
We always assume that $v$ is homogeneous whenever the notation $\bar v$ appears in expressions. 
By $\Pi$ we denote the switch of parity functor.

The Lie superalgebra $\mathfrak{g}=\mathfrak{gl}(n|n)$ is defined to be the endomorphism algebra $\End(V_{\bar{0}}\oplus V_{\bar 1})$, where $\dim V_{\bar 0}=\dim V_{\bar 1}=n$. Then $\mathfrak{g}=\mathfrak{g}_{\bar 0}\oplus \mathfrak{g}_{\bar 1}$, where 
\[ 
\mathfrak{g}_{\bar 0} = \End(V_{\bar 0}) \oplus \End(V_{\bar 1}) 
\qquad 
\mbox{ and } 
\qquad 
\mathfrak{g}_{\bar 1} = \Hom(V_{\bar 0}, V_{\bar 1}) \oplus \Hom(V_{\bar 1},V_{\bar 0}). 
\] 
Let $[x,y] = xy-(-1)^{\bar x \bar y}yx$, where $x$ and $y$ are homogeneous elements of $\mathfrak{g}$, and extend $[\:\:\:,\:\:\:]$ linearly to all of $\mathfrak{g}$. 
By fixing a basis of $V_{\bar 0}$ and $V_{\bar 1}$, the superalgebra $\mathfrak{g}$ can be realized as the set of $2n\times 2n$ matrices, where  
\[ 
\mathfrak{g}_{\bar 0} = 
\left\{  
\begin{pmatrix}
A & 0 \\ 
0 & D \\ 
\end{pmatrix} : 
A, D\in M_{n,n}
\right\}
\mbox{ and }
\mathfrak{g}_{\bar 1} = 
\left\{  
\begin{pmatrix} 
0 & B \\ 
C & 0 \\ 
\end{pmatrix} : 
B,C \in M_{n,n}
\right\}, 
\] 
and  $M_{n,n}$ are $n\times n$ complex matrices. 
Recall that the supertrace is defined by 
\[ 
\mathfrak{str}\begin{pmatrix} A&B\\C&D \end{pmatrix} = \tr(A) - \tr(D). 
\] 

\subsubsection{Periplectic Lie superalgebra}

Let $V$ be an $(n|n)$-dimensional vector superspace equipped with a nondegenerate odd symmetric form
\begin{equation}
\label{eqn:bilinear-form-Pn}
\beta:V\otimes V\to\mathbb C,\quad \beta(v,w)=\beta(w,v), \quad\text{and}\quad \beta(v,w)=0 \quad \text{if} \quad  \bar {v}=\bar{w}.
\end{equation}
Then $\op{End}_{\mathbb C}(V)$ inherits the structure of a vector superspace from $V$. Let $\mathfrak{p}(n)$ be the Lie superalgebra of all $X\in\operatorname{End}_{\mathbb C}(V)$ preserving $\beta$, i.e., $\beta$ satisfies the condition  $$\beta(Xv,w)+(-1)^{\bar{X}\bar{v}}\beta(v,Xw)=0.$$

With respect to a fixed bases for $V$, the matrix of $X\in \mathfrak{p}(n)$ has the form 
$\left(\begin{smallmatrix}
A&B\\
C&-A^t
\end{smallmatrix}
\right)$, 
where $A,B,C$ are $n\times n$ matrices such that $B$ is symmetric and $C$ is antisymmetric.

For the remainder of this manuscript, we will write $\mathfrak{g}:= \mathfrak{p}(n)$.
  Note that $\mathfrak{str}:\mathfrak{g}\rightarrow \mathbb{C}$ is a one-dimensional representation of $\mathfrak{g}$.
  We will also use the $\mathbb Z$-grading $\mathfrak{g}=\mathfrak{g}_{-1}\oplus\mathfrak{g}_0\oplus\mathfrak{g}_1$ where
  $\mathfrak{g}_0=\mathfrak{g}_{\bar 0}\simeq\mathfrak{gl}(n)$, $\mathfrak{g}_{\pm 1} $ is the annihilator of $V_{\bar 0}$ (respectively, $V_{\bar 1}$).
  By $G$ we denote the algebraic supergroup $P(n)$.

\subsection{Root systems}
\label{subsection:root-systems}
For the periplectic Lie superalgebra $\mathfrak{g}$, 
fix the standard Cartan subalgebra $\mathfrak{h}$ of diagonal matrices in $\mathfrak{g}_0$ with its standard dual basis $\{ \varepsilon_1,\ldots, \varepsilon_n \}$.  
So we have a root space decomposition 
$\mathfrak{g} 
= \mathfrak{h}\oplus \left( \bigoplus_{\alpha\in \Delta} \mathfrak{g}_{\alpha}\right)$, 
where 
$\Delta = \Delta(\mathfrak{g}_{-1})\cup \Delta(\mathfrak{g}_0)\cup \Delta(\mathfrak{g}_1)$, and 
\begin{align*} 
\Delta(\mathfrak{g}_0) &= 
\{ 
\varepsilon_i - \varepsilon_j: 1\leq i \not= j \leq n
\},   \\
\Delta(\mathfrak{g}_1) = 
\{ 
\varepsilon_i + \varepsilon_j : 1\leq i\le j \leq n& 
\}, \quad 
\mbox{ and }
\quad 
\Delta(\mathfrak{g}_{-1}) = 
\{ 
-(\varepsilon_i + \varepsilon_j) : 1\leq i< j\leq n 
\}.  
\end{align*}

The set of simple roots is chosen to be 
\[
\Pi= \{ -2\varepsilon_1, \varepsilon_1-\varepsilon_2, \ldots, \varepsilon_{n-1}-\varepsilon_n \}. 
\]  
This implies that 
$\Delta^+(\mathfrak{g}_{0}) = \{ \varepsilon_i-\varepsilon_j:1\leq  i < j\leq n \}$. 
Our Borel subalgebra is then $\mathfrak b_0 \oplus \mathfrak g_{-1}$, where $\mathfrak b_0=\bigoplus_{\alpha\in\Delta^+(\mathfrak g_0)}\mathfrak g_\alpha$ and $\mathfrak g_{-1}=\bigoplus_{\alpha\in \Delta(\mathfrak g_{-1})}\mathfrak g_\alpha$.

Let 
\[ 
\mathcal{R}_{0}^{} = \prod_{\alpha\in\Delta^+(\mathfrak{g}_{0})} (1-e^{-\alpha}), 
\quad 
\mbox{ and }
\quad 
\mathcal{R}_{-1}^{} = \prod_{\alpha\in\Delta(\mathfrak{g}_{-1})} (1-e^{-\alpha}). 
\] 
For $W=S_n$, the Weyl group of the even subalgebra of $\mathfrak p(n)$, $\mathcal{R}_{-1}^{}$ is $W_{}$-invariant and $e^{\rho}\mathcal{R}_{0}^{}$ is $W$-anti-invariant.

\subsection{Weight spaces}
\label{subsection:weight-spaces }
 Let $\mathcal C_n$ be the category of finite-dimensional representation of $\mathfrak {g}$ and $\mathcal F_n$ be the category of finite-dimensional
  representation of $G$. Both are abelian symmetric rigid tensor categories. The latter category is equivalent to the category of finite-dimensional $\mathfrak{g}$-modules, 
  integrable over the underlying algebraic group $G_0=GL(n)$, see \cite{S1}.

The Cartan subalgebra $\mathfrak{h}$ is abelian, 
so it acts locally-finitely on a finite-dimensional $\frak g$-module $M$. 
This yields a decomposition of $M$ as a direct sum of generalized weight spaces 
$M=\oplus_{\lambda\in\mathfrak{h}^*}M_{\lambda}$ where 
$M_{\lambda}=\{ v\in M: (h-\lambda(h))^{m} v \mbox{ for all }h\in \mathfrak{h}\}\not= \{0\}$  for some sufficiently large $m$.  If $M$ is a $G$-module, it is semisimple over $\mathfrak{g}_0$ and hence $\mathfrak{h}$ acts diagonally on $M$.

Suppose that $M=\bigoplus_{\mu \in \mathfrak h ^*}M_\mu$ is weight space decomposition of a $\mathfrak g$-module $M$. Define the character of $M$ as
\[
\ch(M) := \sum_{\mu\in\mathfrak h^*} \dim (M_{\mu})e^{\mu}, 
\] 
while the supercharacter is defined as 
\[ 
\sch(M) := \sum_{\mu\in\mathfrak h^*} \sdim (M_{\mu}) e^{\mu}. 
\]

Weights of modules in the abelian category $\mathcal{F}_n$ of finite-dimensional representations of the periplectic Lie supergroup $P(n)$ are denoted as 
\[
\lambda = (\lambda_1,\ldots, \lambda_n) = \sum_{1\leq i\leq n} \lambda_i \varepsilon_i, \qquad \lambda_i \in \mathbb{Z}. 
\] 
Define the parity of $\lambda$ as $p(\lambda)=\frac{1}{2}\sum_{1\leq i\leq n} \lambda_i\Modulo 2$  
if $\sum_{1\leq i\leq n} \lambda_i$ is even and $p(\lambda)=\frac{1}{2}(\sum_{1\leq i\leq n} \lambda_i+1)\Modulo 2$ if  
$\sum_{1\leq i\leq n} \lambda_i$ is odd. 
Note that the standard ordering of the weights for our choice of positive roots is $\varepsilon_i > \varepsilon_j$ when $i< j$ and $\varepsilon_i < 0$ for all $i$.

A weight $\lambda$ is dominant if and only if $\lambda_1 \geq \lambda_2 \geq \ldots \geq \lambda_n$. 
We will denote $\Lambda_n$ as the set of dominant integral weights. Simple objects in $\mathcal{F}_n$ (up to isomorphism and parity-switch) are parametrized by $\Lambda_n$. Denote by $L(\lambda)$ the simple module with highest weight $\lambda$ with respect to the Borel subalgebra $\mathfrak{b}_0\oplus \mathfrak{g}_{-1}$,  
where the parity is taken such that the parity of the highest weight vector is $p(\lambda)$.

\subsection{Thin Kac modules}  
\label{subsection:thin-Kac} 
Let $V(\lambda)$ be a simple $\mathfrak{g}_0$-module with highest weight $\lambda$ with respect to the fixed Borel $\mathfrak{b}_0$ of $\mathfrak{g}_0$. 
Given a dominant integral weight $\lambda$, the thin Kac module corresponding to $\lambda$ is   
\[
\nabla(\lambda) = \prod\!{}^{n(n-1)/2} \Ind_{\mathfrak{g}_0\oplus \mathfrak{g}_1}^{\mathfrak{g}} V(\lambda-\gamma) \simeq \Coind_{\mathfrak{g}_0\oplus \mathfrak{g}_1}^{\mathfrak{g}} V(\lambda), 
\] 
where $\gamma=\sum_{\alpha\in\Delta(\gg_{-1})}\alpha=\sum_{i=1}^n(1-n)\varepsilon_i$. We will also write $\nabla_{P(n)}(\lambda)$ to specify that the thin Kac module is a representation over the algebraic supergroup $P(n)$.

Let
\[ 
\rho: = \sum_{1\leq i\leq n} (n-i)\varepsilon_i. 
\]
  
\begin{lemma}
\label{lem:superchar-thin-Kac}
The supercharacter of the thin Kac module $\nabla(\lambda)$ with weight $\lambda$ is 
\begin{equation}
\label{eqn:superchar-thin-Kac-Pn}
\sch \nabla(\lambda) = (-1)^{p(\lambda)} 
\frac{\mathcal{R}_{-1}}{
e^{\rho}\mathcal{R}_0} 
\sum_{w\in W}(-1)^{\ell(w)}e^{w(\lambda+ \rho)}.
\end{equation}

\end{lemma}

\begin{proof} 
Since $\nabla(\lambda)\cong \bigwedge 
(\mathfrak{g}_{-1}^*)\otimes V(\lambda)$ as $\mathfrak h$-modules,  
the supercharacter of $\bigwedge(\mathfrak{g}_{-1}^*)$ 
and the character of $V(\lambda)$ are 
\[
\sch \bigwedge (\mathfrak{g}_{-1}^*) = \prod_{\beta\in\Delta(\mathfrak{g}_{-1})}(1-e^{-\beta})   
\quad 
\mbox{ and } 
\quad 
\ch V(\lambda) =  
\left( e^{\rho_0} 
\mathcal{R}_0 
\right)^{-1}
\sum_{w\in W}(-1)^{\ell(w)}w(e^{\lambda+\rho_0}),   
\] 
respectively, 
where $\rho_0=\frac{1}{2}\sum_{\alpha\in\Delta^+(\mathfrak{g}_{0})}\alpha$. 
Since $-\rho_0+w\rho_0=\rho-w\rho$, we obtain \eqref{eqn:superchar-thin-Kac-Pn}. 
\end{proof}

\subsection{Weight diagrams}
Let $\{ \overline{\lambda}_1,\ldots, \overline{\lambda}_n\}\subseteq \mathbb{Z}$ be such that $\lambda+\rho = \sum_{i=1}^n \overline{\lambda}_i \varepsilon_i$.  
The weight diagram $d_{\lambda}$ corresponding to a dominant weight $\lambda$ is the labeling of the line of integers by symbols $\bullet$ and $\circ$, where $i$ has label $\bullet$ if $i\in \{ \overline{\lambda}_1,\ldots, \overline{\lambda}_n\}$, 
and $\circ$ if $i\not\in \{ \overline{\lambda}_1,\ldots, \overline{\lambda}_n\}$. 
For example, $ds_{0}$ is 
\[ 
\xymatrix@-1pc{
\ldots & \underset{-1}{\circ} & \underset{0}{\bullet} & 
\underset{1}{\bullet} & \ldots & \underset{n-1}{\bullet} & \underset{n}{\circ} &\underset{n+1}{\circ} & \ldots 
}
\] 
and $d_{-3\varepsilon_n-\varepsilon_{n-1}}$ is 
\[ 
\xymatrix@-1pc{
\ldots & \underset{-4}{\circ} & \underset{-3}{\bullet} 
& \underset{-2}{\circ} & \underset{-1}{\circ} & \underset{0}{\bullet} & 
\underset{1}{\circ}  & \underset{2}{\bullet} & \underset{3}{\bullet} &\ldots  &\underset{n-1}{\bullet}. 
}
\]

Note that $\lambda \leq \mu$ if and only if $\lambda_i \geq \mu_i$ for each $i$. In terms of weight diagrams,  the $i$-th black ball in $d_{\lambda}$ (counted from left) lies further to the right of the $i$-th black ball of $d_{\mu}$.

\subsection{The Grothendieck ring}
\label{subsubsection:reduced-Gr-Pn}
Let $K(P(n))$ be the Grothendieck ring of  $\mathcal{F}_n$, and define 
\begin{equation}
\label{eqn:red-Groth-ring}
J(P(n)) = K(P(n))/\langle [M] + [\Pi M]: M \in \mathcal{F}_n \rangle.  
\end{equation}
The ring $J(P(n))$ is isomorphic to the reduced Grothendieck ring 
$K(P(n))/\langle [M]-[\Pi M] : M\in \mathcal{F}_n \rangle$,  
with the isomorphism given by $[L(\lambda)]\mapsto (-1)^{p(\lambda)}[L(\lambda)]$, 
where $L(\lambda)$ is the simple module of highest weight $\lambda$.

One may identify $J(P(n))$ as the ring of supercharacters as follows: 
let $\Lambda \subseteq \mathfrak{h}^*$ be the abelian group of integral weights of $\mathfrak{g}_0$ and $W$ be the Weyl group of $\mathfrak{g}_0$. 
The supercharacter function  $\sch: J(P(n))\rightarrow  \Span_ {\mathbb Z }\{e^\lambda : \lambda\in\Lambda\}$,  sends $[M]\mapsto \sch(M)$. 
 Since isomorphic modules have the same supercharacter and $\sch(M) = -\sch \Pi M$, $\sch$ is well-defined. Furthermore, $\sch$ is injective since two irreducible modules have the same character if and only if they are isomorphic.

 Throughout this manuscript, we will also write $x_i:= e^{\varepsilon_i}$ for $1\leq i\leq n$. 
The ring $$J_n:= \{f\in \mathbb{Z}[x_1^{\pm 1},\ldots, x_n^{\pm 1}]^{S_n}: f|_{x_i = x_j^{-1}=t} \mbox{ is independent of }t \mbox{ for }i\not=j \}$$ is then identified with a subring of $\Span_ {\mathbb Z }\{e^\lambda : \lambda\in\Lambda\}$. 

The following lemma is proved using restriction  to  subalgebras of the form $\mathfrak g_{-\alpha}\oplus \mathfrak h \oplus \mathfrak g_\alpha$ where $\alpha$ is an odd root and $2\alpha$ is not a root.

\begin{lemma}[{\cite[Prop. 4.3]{MR2776360}}]
\label{lemma:Pn-superchar-containment}
We have 
\[ J(P(n))\subseteq J_n. 
\]  
\end{lemma}

\subsection{Translation functors} 
\label{subsection:Translation-functors} 
Let  $\Theta' = -\otimes V:\mathcal{F}_n\rightarrow \mathcal{F}_n$ be an endofunctor. 
Consider the involutive anti-automorphism $\sigma:\mathfrak{gl}(n|n)\rightarrow \mathfrak{gl}(n|n)$ defined as 
\[ 
\begin{pmatrix}
A & B \\ 
C & D \\ 
\end{pmatrix}^{\sigma} := 
\begin{pmatrix}
-D^t & B^t \\ 
-C^t & -A^t \\ 
\end{pmatrix}. 
\] 
One can see that $\mathfrak{p}(n)=\{ x\in \mathfrak{gl}(n|n): x^{\sigma} = x\}$, and we set  
$\mathfrak{p}(n)^{\perp} := \{ x\in \mathfrak{gl}(n|n): x^{\sigma} = -x \}$. 

Since $\mathfrak{p}(n)$ and $\mathfrak{p}(n)^{\perp}$ form maximal isotropic subspaces with respect to the form $\mathfrak{str}XY$ we obtain
a nondegenerate bilinear $\mathfrak{p}(n)$-invariant pairing 
  $\langle \:\cdot\: , \:\cdot\: \rangle: \mathfrak{p}(n)\otimes \mathfrak{p}(n)^{\perp}\rightarrow \mathbb{C}$. 
Choose $\mathbb{Z}$-homogeneous bases $\{ X_i\}$ in $\mathfrak{p}(n)$ and $\{ X^i\}$ in $\mathfrak{p}(n)^{\perp}$ such that $\langle X^i , X_j \rangle = \delta_{ij}$. 
We define the fake Casimir element as 
\[
\Omega := 2 \sum_{i=1}^n X_i \otimes X^i \in \mathfrak{p}(n)\otimes \mathfrak{p}(n)^{\perp} \subseteq \mathfrak{p}(n) \otimes \mathfrak{gl}(n|n). 
\]
 
Given a $\mathfrak{p}(n)$-module $M$, let $\Omega_M:M\otimes V\rightarrow M\otimes V$ be the linear map 
\[ 
\Omega_M(m\otimes v) = 2 \sum_{1\leq i \leq n} (-1)^{\overline{X}_i \overline{m}} X_i m \otimes X^i v,
\]  
where $m\in M$ and $v\in V$ are homogeneous. By \cite[Lemma 4.1.4]{BDEHHILNSS}, we see that $\Omega_M$ commutes with the action of $\mathfrak{p}(n)$ on $M\otimes V$ for any $\mathfrak{p}(n)$-module $M$. 

For $k\in \mathbb{C}$, define a functor $\Theta_k':\mathcal{F}_n\rightarrow \mathcal{F}_n$ as $\Theta'=-\otimes V$ followed by the projection onto the generalized $k$-eigenspace for $\Omega$, i.e., 
\[
\Theta_k'(M) := \bigcup_{m> 0} \ker(\Omega-k\: \text{Id})^m \Big|_{M\otimes V}. 
\] 
Since $\Theta_k'=0$ if $k\notin\mathbb{Z}$, we set $\Theta' = \bigoplus_{k\in \mathbb{Z}} \Theta_k'$, 
and $\Theta_k := \prod^{k} \Theta_k'$ when $k\in \mathbb{Z}$. 
The endofunctors 
 $\Theta_k$ of $\mathcal{F}_n$ for $k\in \mathbb{Z}$ are exact.

The following is \cite[Prop. 5.2.2]{BDEHHILNSS}.  

\begin{proposition}[Translation of thin Kac modules]  \label{smallKac} Let $k\in\mathbb{Z}$. Then 
\begin{enumerate}
 \item\label{item:right-trans-fnr} 
    $\Theta'_k \nabla(\lambda) = \nabla(\mu'')$ if  $d_{\lambda}$ looks as follows at positions $k-1, k, k+1$, with $d_{\mu''}$ displayed underneath:
 \begin{align*} 
  d_\lambda= \xymatrix{  &\underset{k-1}{\bullet}  &\underset{k}{\bullet}  &\underset{k+1}{\circ} }  \\
 d_{\mu''}=\xymatrix{  &\underset{k-1}{\bullet}  &\underset{k}{\circ}  &\underset{k+1}{\bullet} }
\end{align*}
\item\label{item:left-trans-fnr}  
    $\Theta'_k \nabla(\lambda)=\Pi\nabla(\mu')$ if  $d_{\lambda}$ looks as follows at positions $k-1, k, k+1$, with $d_{\mu'}$ displayed underneath:
\begin{align*} 
 d_{\lambda}= \xymatrix{  &\underset{k-1}{\circ}  &\underset{k}{\bullet}  &\underset{k+1}{\bullet}   }  \\
  d_{ \mu'}= \xymatrix{  &\underset{k-1}{\bullet}  &\underset{k}{\circ}  &\underset{k+1}{\bullet}  }
\end{align*}
\item\label{item:SES-trans-fnr} 
    In case $d_\lambda$ looks locally at positions $k-1,k,k+1$ as below, there is a short exact sequence
\[
0 \rightarrow \nabla(\mu'') \rightarrow \Theta'_k \nabla(\lambda) 
\rightarrow  \Pi\nabla(\mu') \rightarrow 0, 
\] 
where $d_{\mu'}$ and $d_{\mu''}$ are obtained from $d_\lambda$  
by moving one black ball away from position $k$ 
(to position $k-1$, respectively, $k+1$) as follow:
\begin{align*} 
  d_\lambda= \xymatrix{  &\underset{k-1}{\circ}  &\underset{k}{\bullet}  &\underset{k+1}{\circ}   }  \\
   d_{\mu'} = \xymatrix{  &\underset{k-1}{\bullet}  &\underset{k}{\circ}  &\underset{k+1}{\circ}  }\\
    d_{\mu''}=\xymatrix{  &\underset{k-1}{\circ}  &\underset{k}{\circ}  &\underset{k+1}{\bullet}    
     }
\end{align*}
\item  
    $\Theta'_k\nabla(\lambda) =0$ in all other cases.
\end{enumerate}
\end{proposition}
\subsection{Parabolic induction} \label{parabolic induction}
We consider the standard scalar product on $\mathfrak h^*$ such that $(\varepsilon_i,\varepsilon_j)=\delta_{i,j}$. Let $\gamma$ be some weight. Set
$$\mathfrak{k}:=\mathfrak h\oplus \bigoplus_{(\alpha,\gamma)=0}\mathfrak{g}_{\alpha},\quad  \mathfrak{r:}=\bigoplus_{(\alpha,\gamma)>0}\mathfrak{g}_{\alpha},\quad  \mathfrak{q}:=\mathfrak{k}\oplus \mathfrak{r}.$$
The subalgebra $\mathfrak{q}$ is called a parabolic subalgebra of $\mathfrak{g}$ with the Levi subalgebra $\mathfrak{k}$ and the nilpotent radical $\mathfrak{r}$.
If $Q$ is the corresponding parabolic subgroup of $G$, then $G/Q$ is a generalized flag supervariety.

Let $\lambda$ be a weight such that $\lambda(\mathfrak{h}\cap [\mathfrak{k},\mathfrak{k}])=0$. 
We denote by $\mathcal{O}(-\lambda)$ the line bundle on $G/Q$ induced by the one dimensional representation of $Q$ with weight $-\lambda$. 
Set 
\[ 
\mathcal{E}(\lambda) = \sum_i (-1)^i \sch H^i(G/P,\mathcal{O}(-\lambda)). 
\] 
By definition $\mathcal E(\lambda)$ is in $J(P(n))$. By \cite[Prop. 1]{MR2734963}, 
\[  
\mathcal{E}(\lambda)  
= \frac{1}{  e^\rho 
\mathcal{R}_0} \sum_{w\in W} 
(-1)^w w\left(e^{\lambda+\rho}\prod_{\alpha\in \Delta_1(\mathfrak{r})} (1-e^{-\alpha})\right) 
\] 
where $\Delta_1(\mathfrak{r}):=\{\alpha\in \Delta_{1} : (\alpha,\gamma)>0  \}$.  
Since $e^\rho\mathcal{R}_0$ is W-anti-invariant, this can be written as 
\[ 
\mathcal{E}(\lambda) 
= \sum_{w\in W} w\left( 
e^{\lambda} \frac{\prod_{\alpha\in \Delta_1(\mathfrak{r})}(1-e^{-\alpha})  }{
\prod_{\alpha\in \Delta_0^+} (1-e^{-\alpha})}
\right). 
\]

\section{Duflo--Serganova homomorphism for $P(n)$}
\label{section:DS-functor}
We show that the Duflo--Serganova functor induces a homomorphism between the rings of supercharacters, and discuss the kernel of this homomorphism (cf. \cite{hoyt2016grothendieck}). 

\subsection{The $ds_n$ homomorphism}
\label{subsection:well-defined-kernel}
Let $x\in \mathfrak{g}_1\oplus \mathfrak g_{-1}$ such that $[x,x]=0$. Then $x^2$ is zero in $U(\mathfrak g)$ and for every $\mathfrak g$-module $M$, we define 
\[ 
M_x=\ker_M x/ xM, 
\] 
and similarly,
\[ 
\mathfrak g_x=\mathfrak g^x /[x,\mathfrak g], 
\] 
where 
$\mathfrak{g}^x=\{ g\in \mathfrak{g}: [x,g]=0\}$.
By \cite[Lemma 6.2]{duflo2005associated}, 
$M_x$ carries a natural $\mathfrak{g}_x$-module structure. 
Moreover, the
  Duflo--Serganova functor $DS_x:M\mapsto M_x$
  is a symmetric monoidal functor from the category of $\mathfrak{g}$-modules to the category of $\mathfrak{g}_x$-modules.
We consider a special case of  the DS functor:
\[
DS_{n}: \mathcal{F}_n \rightarrow \mathcal{F}_{n-2},\quad DS_x(M)=M_x, 
\]
defined as follows.
Suppose 
$x=x_\beta$ for $\beta\in \Delta(\mathfrak g_{-1})$. 
By \cite[Lemma 5.1.2]{ES-Deligne-perip}, 
$\mathfrak g_x \cong \mathfrak p(n-2)$. We embed $\mathfrak g_x$ in $\mathfrak g$ such that the Cartan subalgebra $\mathfrak h_x$ of  $\mathfrak g_x$ is contained in 
$\{h\in \mathfrak h : \beta(h)=0\}$. 
By \cite[Sec. 3]{hoyt2016grothendieck}, 
the map 
\begin{equation}
\label{eqn:induced-ds-map}
ds_{x}: J(P(n))\rightarrow J(P(n-2)), 
\quad 
\mbox{ where } 
\:\:
ds_{x}(\sch M)=\sch(DS_x(M)), 
\end{equation}
 is well-defined and equal to 
$(\sch M)|_{h_{x}}$. 
Since $\beta=-\varepsilon_i-\varepsilon_j$ and $f\in J(P(n))$ satisfies that 
$f|_{x_i=x_j^{-1}=t}$ is independent of $t$, we get that 
$f|_{h_{x}}=f|_{\{ h\in \mathfrak{h}: \varepsilon_i(h)=-\varepsilon_j(h)\}}=f|_{x_i=x_j^{-1}}$.

Let $-\varepsilon_i-\varepsilon_j\in \Delta(\mathfrak{g}_{-1})$, 
and define  
\begin{equation}
ds_n:J(P(n))\rightarrow J(P(n-2)),\quad ds_n=p\circ ds_{x_{\alpha}}, 
\end{equation} 
where $p$ is a bijection $p:\{1,\ldots, n \}\setminus \{ i,j\}\rightarrow \{ 1,\ldots, n-2\}$. 
Since the elements in $J(P(n))$ are $S_n$-invariant, 
$ds_n$ is independent of $i$ and $j$.  
Moreover, the map $ds_n$ extends naturally to   $J_n$  
by the evaluation
$ds_n(f)=f|_{x_{n-1}=x_n^{-1}=t}$ (eventually we show that $J(P(n))=J_n$ but for now $J_n\subseteq J(P(n))$).

We define 
\begin{equation}
\label{eqn:dsnk}
ds_n^{(k)} := ds_{n-2k+2} \circ \cdots \circ ds_{n}. 
\end{equation} 
Note that applying $ds_n^{(k)}$ is the same as applying $ds_x$ for $x$ of higher rank.

\subsection{The kernel of $ds_n$}
\label{subsection:kernel-ds-map} 
The following proposition is a straightforward generalization of \cite[Thm. 17]{hoyt2016grothendieck}. 
  
\begin{proposition}
\label{lemma:kernel-DS-Pn}
The kernel of $ds_{n}$ is spanned by the supercharacters of thin Kac modules. 
\end{proposition}
\begin{proof}
Suppose $f\in \ker ds_n$. Then $f$ is divisible by $1-x_{n-1}x_n$. 
Since $f$ is $W$-invariant,   $f$ is also divisible by $\prod_{i<j}(1-x_ix_j)$ and hence $$f=\prod_{i<j}(1-x_ix_j)g=\mathcal{R}_{-1}\cdot g,$$ where $g$ is also $W$-invariant.  
Write $g$ as a linear combinations of Schur functions 
\[ 
g=\sum_{\lambda\in\mathfrak{h}^*}^{\finite} a_{\lambda} e^{-\rho}\mathcal{R}_0^{-1}\sum_{w\in W}(\sgn w)w(e^{\lambda+\rho}).
\]   
Thus $f=\sum_{\lambda\in \mathfrak{h}^*}^{\finite} a_{\lambda}\nabla(\lambda)$. 
\end{proof}
\subsection{Translation functors and $DS_n$} 
We will need the following statement, see  \cite[Corollary 3.0.2]{ES-KW-perip}.
\begin{lemma}
\label{commute} 
The functor $DS_n$ commutes with translations functors $\Theta'_k$. 
  \end{lemma}

\begin{corollary}
\label{lemma:translation-functor-image}
If 
$[M]\in \Im ds_{n}$, then 
$[\Theta_i(M)]\in \Im ds_{n}$ for every translation functor $\Theta_i$. 
\end{corollary}

\begin{proof}
Suppose that
$DS_n(A)=M$ for some finite-dimensional $P(n)$-module $A$. 
By Lemma~\ref{commute}, 
we have 
$ds_n([\Theta_i(A)]) 
= [DS_{n}(\Theta_i(A))] 
= [\Theta_i (DS_n (A))]
= [\Theta_i(M)]$.
\end{proof}

\section{Surjectivity of the Duflo-Serganova map for $P(n)$}
\label{section:main-thm-proof} 
As explained in Proposition \ref{lemma:kernel-DS-Pn}, the kernel of $ds_n$ is well-understood. We now turn to discuss the image of $ds_n$ and prove 
 the following theorem.

\begin{theorem}\label{thm:ds-map-surj}
The map $ds_n:J(P(n))\rightarrow J(P(n-2))$ is surjective. 
\end{theorem}

We prove Theorem~\ref{thm:ds-map-surj} in three steps. 
We first show that if $[\nabla(0)]$ is in the image of $ds_n^{(k)}$ for some $k\geq 0$ (see \eqref{eqn:dsnk} for the definition), then $[\nabla(\mu)]$ is also in the image of $ds_n^{(k)}$ for every $\mu$. 
We then show that $[\nabla(0)]$ is in the image of $ds_n$ by explicitly constructing its preimage.  
Finally, we show that the fact that $[\nabla(\mu)]$ is in the image of $ds_n^{(k)}$ for every $\mu$ implies that the map is surjective.

\subsection{Thin Kac modules are in the image of $ds_n$}
 We prove the following proposition using the action of translations functors $\Theta_k$ on thin Kac modules. 
 
\begin{proposition}
\label{prop:all-thin-Kac-image}
If $[\nabla(0)]\in \Im ds_n^{(k)}$ for some $k \geq 0$, then $[\nabla(\mu)]\in \Im ds_n^{(k)}$ for all $\mu\in \Lambda_{n}$. 
\end{proposition}

\begin{proof}[Proof of Proposition~\ref{prop:all-thin-Kac-image}]
Assume that $[\nabla(0)]\in \Im ds_n^{(k)}$ for some $k\geq 0$. 
Consider $\Theta_0(\nabla(0))=\nabla(-\varepsilon_n)$. 
By Corollary~\ref{lemma:translation-functor-image}, $[\nabla(-\varepsilon_n)]$ is also in the image. 
Now, apply $\Theta_{-1}$ to $\nabla(-\varepsilon_n)$. Then 
$[\Theta_{-1}(\nabla(-\varepsilon_n))] 
= -[\nabla(-2\varepsilon_n)] + [\nabla(0)]$ by Proposition~\ref{smallKac} \eqref{item:SES-trans-fnr}.
Since $[\Theta_{-1}(\nabla(-\varepsilon_n))] -[\nabla(0)]$ is in the image, 
$-[\nabla(-2\varepsilon_n)]$ is also in the image. 
We can inductively apply $\Theta_{-k_n}$, 
where $k_n\geq 1$, 
to $\nabla(-k_n\varepsilon_n)$ to obtain 
\[ 
[\Theta_{-k_n}(\nabla(- k_n \varepsilon_n))]
=- [\nabla(-(k_n + 1)\varepsilon_n)] + [\nabla(-(k_n - 1)\varepsilon_n)], 
\]
 Thus $[\nabla(-k_n\varepsilon_n)]$ is in the image for every $k_n\geq 1$.

Now assume that $\nabla\left( -\sum_{i=r}^n k_i \varepsilon_i \right)$ 
is in the image for all $k_n \geq \ldots \geq  k_r\ge n-r$, $r\ge 2$.
In a similar way, 
we apply  $\Theta_{n-r+1}$ 
to $\nabla\left( -\sum_{i=r}^n k_i \varepsilon_i \right)$ 
to obtain that 
$\nabla\left( -\varepsilon_{r-1}-\sum_{i=r}^n k_i \varepsilon_i \right)$  
is also in the image. 
Now suppose that 
$\nabla\left( -j\varepsilon_{r-1}-\sum_{i=r}^n k_i \varepsilon_i \right)$ is in the image for some $1\le j < k_r$.  
We apply 
$\Theta_{n-r+1-j}$ to 
$\nabla\left( -j\varepsilon_{r-1}-\sum_{i=r}^n k_i \varepsilon_i \right)$ 
to obtain that 
$\nabla\left( -(j+1)\varepsilon_{r-1}-\sum_{i=r}^n k_i \varepsilon_i \right)$ 
is also in the image.

    Thus, we obtain that $\nabla\left(-\sum_{i=2}^n k_i\varepsilon_i  \right)$ is in the image for all $k_n\ge \ldots \ge k_2\ge 1$. Finally, we can obtain all other thin Kac modules via tensoring with powers of the supertrace representation. Indeed, $L\left(k_1\sum_{i=1}^n \varepsilon_i \right)$ is in the image because $\sch L\left(k_1\sum_{i=1}^n \varepsilon_i \right)=\left(x_1\cdots x_n\right)^{k_1}$  and 
    $ds_n\left( [L\left(k_1\sum_{i=1}^{n+2} \varepsilon_i \right)] \right) 
    = [L\left(k_1\sum_{i=1}^n \varepsilon_i \right)]$. Since 
    \[
    \nabla(\mu)\otimes L\left(k_1\sum_{i=1}^{n} \varepsilon_i \right) =\nabla\left(\mu+k_1\sum_{i=1}^{n} \varepsilon_i  \right),
    \]
 the assertion follows. 
\end{proof}

We now show that $[\nabla(0)]$ is in fact in the image of $ds_{n}^{(k)}$. We construct the preimage using parbolic induction.

\begin{proposition}\label{nabla(0) is in the image}
One has $[\nabla(0)]\in \Im ds_{n}^{(k)}$. 
\end{proposition}

\begin{proof} 
Choose the parabolic subalgebra $\mathfrak{q}$ associated with $\gamma=-\sum_{l=2k+1}^n\varepsilon_l$ (see Section \ref{parabolic induction}). Then we have
$\mathfrak{k}\simeq \mathfrak{p}(2k)+\mathfrak{gl}(n-2k)$ and  
\begin{align*}
\Delta_0(\mathfrak{r}) &= \{ \varepsilon_i -\varepsilon_j: 1\leq i\leq 2k < j\leq n \}, \\ 
\Delta_1(\mathfrak{r}) &= \{ -\varepsilon_i-\varepsilon_j: i< j, \:\: 1\leq i\leq n, \:\: 2k < j\leq n \}. 
\end{align*}
We set $\lambda = a(\varepsilon_1+ \ldots + \varepsilon_{2k})$ for $a\in \mathbb Z$ 
and define $\widetilde{ds}_n^{(k)}$ 
as an evaluation at $x_i = t_i$, $x_{i+k}=t_i^{-1}$ for $i=1,\ldots, k$  
(namely the same as $ds_n^{(k)}$ up to a permutation). 
So 
\[ 
\mathcal{E}(\lambda) = 
\sum_{w\in S_n} w\left(
e^{\lambda}
\frac{\prod_{1\leq i\leq n, 2k < j \leq n, i< j} (1-x_i x_j)}{
\prod_{1\leq i<j\leq n} (1-x_i^{-1}x_j)
}
\right). 
\] 
Note that if $w\not\in S_{2k}\times S_{n-2k}$, then there exists $i\leq k$ such that $w^{-1}(\{ i,i+k\})\not\in \{ 1,\ldots, 2k\}$. 
This implies that for any $w\not\in S_{2k}\times S_{n-2k}$, we have 
\[
\widetilde{ds}_n^{(k)} \left(w\left(\prod_{1\leq i\leq n,2k<j\leq n} (1-x_ix_j)\right)\right)=0. 
\] 
Thus, we have 
\begin{equation}
\label{eqn:E-lambda-sch-xi}
\widetilde{ds}_n^{(k)} (\mathcal{E}(\lambda)) = 
\sum_{w\in S_{2k}\times S_{n-2k}} 
\widetilde{ds}_n^{(k)} 
\left(w\left(
e^{\lambda} 
\frac{\prod_{1\leq i\leq n,2k<j\leq n,i<j} (1-x_ix_j)}{\prod_{1\leq i<j\leq n} (1-x_i^{-1}x_j)}
\right)\right). 
\end{equation}
Now, we notice that $\lambda$ is $S_{2k}\times S_{n-2k}$-invariant and 
$\widetilde{ds}_n^{(k)}(e^{\lambda})=1$. 
Moreover, set
\[ 
A := \frac{\prod_{1\leq i\leq 2k < j\leq n} (1-x_ix_j)}{
\prod_{1\leq i\leq 2k < j\leq n} (1-x_i^{-1}x_j)  
},
\] 
then $A$ is also $S_{2k}\times S_{n-2k}$-invariant, and 
$\widetilde{ds}_n^{(k)} (A)=1$. 
We can further simplify \eqref{eqn:E-lambda-sch-xi} as 
\[ 
\widetilde{ds}_n^{(k)}(\mathcal{E}(\lambda)) 
= \sum_{w\in S_{2k}\times S_{n-2k}} 
\widetilde{ds}_n^{(k)}\left(w\left(
\frac{\prod_{2k < i < j \leq n} (1-x_ix_j)}{
\prod_{1\leq i<j\leq 2k} (1-x_i^{-1}x_j) 
\prod_{2k < i < j\leq n} (1-x_i^{-1}x_j)}
\right)\right). 
\] 
Finally, the latter expression can be rewritten as 
\[
\widetilde{ds}_n^{(k)} \left( 
\left( 
\sum_{u\in S_{2k}} u\left(
\frac{1}{\prod_{1\leq i<j\leq 2k}(1-x_i^{-1}x_j)} 
\right) 
\right) 
\left( 
\sum_{v\in S_{n-2k}} 
v\left(\frac{\prod_{2k < i < j \leq n}(1-x_ix_j)}{\prod_{2k< i < j\leq n}(1-x_i^{-1}x_j)} 
\right)
\right)  
\right).  
\] 
By the denominator identity of ${\mathfrak{sl}(2k)}$ and 
${\mathfrak{sl}(n-2k)}$, we have that 
\[ 
\sum_{u\in S_{2k}} u
\left(\frac{1}{\prod_{1\leq i < j\leq 2k}(1-x_i^{-1}x_j)}\right) 
= \sum_{v\in S_{n-2k}} 
v\left(\frac{1}{\prod_{2k < i < j\leq n}(1-x_i^{-1}x_j)}\right) = 1. 
\] 
So we finally get 
\[
\widetilde{ds}_n^{(k)} (\mathcal{E}(\lambda)) = \sum_{v\in S_{n-2k}} 
v\left( 
\frac{\prod_{2k < i < j\leq n}(1-x_i x_j) }{\prod_{2k < i < j \leq n}(1-x_i^{-1}x_j)} 
\right) 
= \prod_{2k < i < j\leq n} (1-x_ix_j), 
\] 
and $ds_n^{(k)}  (\mathcal{E}(\lambda)) =\prod_{1\leq i<j\leq 2k-i}(1-x_ix_j)=\sch \nabla(0)$, 
as desired. 
\end{proof}

\subsection{Surjectivity of the $ds_n$ map}

The following proposition concludes the proof of  Theorem~\ref{thm:ds-map-surj}.

 \begin{proposition}
 \label{prop:surjectivity}
 If $\Span\{ [\nabla(\lambda)]: \lambda\in \Lambda_{n-2} \} \subseteq \Im ds_{n}$, then $J(P(n-2))\subseteq \Im ds_{n}$. 
 \end{proposition}

Consider 
\[ 
J\left(P(n)\right)\overset{ds_{n}}{\longrightarrow} 
J(P(n-2))\overset{ds_{n-2}}{\longrightarrow} 
J(P(n-4))\overset{ds_{n-4}}{\longrightarrow} 
\ldots 
\overset{ds_{x}}{\longrightarrow} 
J\left(P\left(n-2\left\lfloor \tfrac{n}{2} \right\rfloor\right)\right), 
\] 
where $x=n-2\left\lfloor \tfrac{n}{2} \right\rfloor+2$. 
Then  
there is a filtration 
\begin{equation}\label{eqn:filtration-kernels}
0 = \ker ds_n^{(0)} \subseteq \ker ds_n^{(1)} \subseteq \ker ds_n^{(2)} \subseteq \ldots \subseteq \ker ds_n^{\left( \lfloor \frac{n}{2}\rfloor\right) } = J(P(n)),  
\end{equation} 
 with  
 \[ 
J(P(n)) = \bigcup_{k=0}^{\lfloor \frac{n}{2} \rfloor} \ker ds_n^{(k)}. 
\]  
We write the associated graded of $J(P(n))$ 
with respect to this filtration, namely, 
\[ 
\Gr J(P(n)) = \bigoplus_{k=1}^{\lfloor \frac{n}{2} \rfloor} \overline{J}_n^{k}, 
\]
where $\overline{J}_n^{k}:= \ker ds_n^{(k)}/\ker ds_n^{(k-1)}$.
Let $\overline{ds}_n:  \Gr J(P(n)) \rightarrow \Gr J(P(n-2))$   and  
\[  
\overline{ds}_n^{(k)}: 
\overline{J}_n^k \hookrightarrow 
\overline{J}_{n-2}^{k-1}
\] 
be the corresponding maps. 
Note that $\overline{ds}_n^{(1)}$ is the zero map. 

\begin{remark}
\label{remark:filtration-HPS}
 The filtration~\eqref{eqn:filtration-kernels} is also used in \cite[Prop. 42]{HPS}. 
\end{remark}

We now prove Proposition~\ref{prop:surjectivity}. 
 
 \begin{proof} 
 It is enough to prove the statement for the associated graded, namely,  
 if $\Span\{ [\nabla(\lambda)] \}_{\lambda\in \Lambda_{n-2}} \subseteq \Im \overline{ds}_{n}$, then $J(P(n-2))\subseteq \Im \overline{ds}_{n}$. 
 We will consider the cases when $n$ is even and odd separately.  
 
We have the following diagram of maps when $n$ is even 
 \[
\xymatrix@-2pc{ 
\stackrel{\Gr J(P(n))}{} & \stackrel{\Gr J(P(n-2))}{} & \ldots & 
\stackrel{\Gr J(P(6))}{} & \stackrel{\Gr J(P(4))}{}  & 
\stackrel{\Gr J(P(2))}{} & \stackrel{\Gr J(P(0))}{}  \\ 
\overset{\overline{J}_n^{(n/2)+1}}{\bullet} \ar[dr]^{\overline{ds}_n^{((n/2)+1)}}  & & & & & & \\ 
\overset{\overline{J}_n^{n/2}}{\bullet} \ar[dr]^{\overline{ds}_n^{(n/2)}} &  \overset{\overline{J}_{n-2}^{n/2}}{\bullet}  &   & & & & \\ 
\vdots &  \overset{\overline{J}_{n-2}^{(n/2)-1}}{\bullet}  & \ddots & & & & \\ 
\overset{\overline{J}_{n}^k}{\bullet} \ar[dr]^{\overline{ds}_{n}^{(k)}} &  \vdots & \ddots & \overset{\:\:\:\:\:\overline{J}_6^4\:\:\:\:\:}{\bullet} \ar[dr]^{\overline{ds}_6^{(4)}} & & & \\ 
\vdots &  \overset{\overline{J}_{n-2}^{k-1}}{\bullet} & \ddots & \overset{\overline{J}_6^3}{\bullet}  
\ar[dr]^{\overline{ds}_6^{(3)}} & \overset{\:\:\:\:\:\overline{J}_4^3\:\:\:\:\:}{\bullet} \ar[dr]^{\overline{ds}_4^{(3)}} & &  \\ 
\overset{\overline{J}_{n}^2}{\bullet} 
\ar[dr]^{\overline{ds}_n^{(2)}} & \vdots  & \ddots & \overset{\overline{J}_6^2}{\bullet} \ar[dr]^{\overline{ds}_6^{(2)}} &
\overset{\overline{J}_4^2}{\bullet} 
\ar[dr]^{\overline{ds}_4^{(2)}} & \overset{\:\:\:\:\:\overline{J}_2^2\:\:\:\:\:}{\bullet} \ar[dr]^{\overline{ds}_2^{(2)}} & \\ 
\overset{\overline{J}_{n}^1}{\bullet} & \overset{\overline{J}_{n-2}^1}{\bullet} & & \overset{\overline{J}_6^1}{\bullet}  &\overset{\overline{J}_4^1}{\bullet} &\overset{\overline{J}_2^1}{\bullet}  & \overset{\:\:\:\:\:\overline{J}_0^1,\:\:\:\:\:}{\bullet} \\ 
}
\] 
and we have a similar diagram when $n$ is odd, with $\Gr J(P(n))=\bigoplus_{k=1}^{(n+1)/2} \overline{J}_n^k$.

We show that all the arrows in the diagram are bijective.
Since $\overline{J}_n^{k}= \ker ds_n^{(k)}/\ker ds_n^{(k-1)}$, 
the map 
$\overline{ds}_n^{(k)}$ is injective for every $k$ and $n$. 
It remains to show that $\overline{ds}_n^{(k)}$ is surjective.  
We prove it by induction on $n$, separately for even and odd $n$.

For $n=2$ , note first that $J(P(0))=\mathbb{Z}$. 
Since   
the map $\overline{ds}_{2}^{}: J(P(2))/\ker ds_2\rightarrow J(P(0))$  sends the supercharacter of the trivial representation to the supercharacter of the trivial representation, i.e., $\overline{ds}_{2}^{} (1)= 1$, we have that $\overline{ds}_{2}^{}$ is surjective, and that is an isomorphism of vector spaces. 
For $n=3$, note that $J(P(1))=\mathbb{Z}[x_1^{\pm 1}]$. 
Since $\overline{ds}_{3}^{}: \overline{J}_3^{2}\rightarrow \overline{J}_1^1$ maps 
$\overline{ds}_{3}^{}(x_1 x_2 x_3)=x_1$, it is an isomorphism.

Now suppose that we proved the statement up to $n-2$. We will prove it for $n$. 
Namely, we show that the maps in the leftmost column in the above diagrams are surjective. The lowest map in the diagram $\overline{ds}_{n}^{(2)}: \overline{J}_{n}^2 \hookrightarrow 
\overline{J}_{n-2}^{1}$ is surjective by Theorem~\ref{thm:ds-map-surj}.   
Consider 
$\overline{ds}_{n}^{(k)}: \overline{J}_{n}^k \hookrightarrow 
\overline{J}_{n-2}^{k-1}$, where 
$2< k\leq \lfloor \frac{n}{2} \rfloor+1$.

First, consider the map $g$ in the diagram below 
\[ 
\xymatrix@-1pc{
\overline{J}_{n}^{k+1} \ar[rrr]^{
\overline{ds}_{n}^{(k+1)}} \ar@/_8mm/[rrrrrrrr]_g & & & 
\overline{J}_{n-2}^k 
\ar[rrrrr]^{
\overline{ds}_{2}^{(2)} \circ \cdots \circ\overline{ds}_{n-2}^{(k)}} & & & & & \overline{J}_{n-2k}^1, 
}
\] 
where $g= \overline{ds}_{2}^{(2)} \circ \cdots \circ\overline{ds}_{n-2}^{(k)}\circ\overline{ds}_{n}^{(k+1)}$. 
The map $g$ is surjective since
when $n$ is even,  $g(1)=1$, 
and when $n$ is odd, $g(x_1\cdots x_n)=x_1$.  
By induction hypothesis, 
$\overline{ds}_{2}^{(2)} \circ \cdots \circ\overline{ds}_{n-2}^{(k)}$ is  bijective. 
Thus, $\overline{ds}_{n}^{(k+1)}$ is a surjection. 
\end{proof}

\section{Grothendieck rings of periplectic Lie superalgebras}
\label{section:Groth-ring-superalgebra}

\subsection{Proof of Theorem~\ref{thm:reduced-GR-periplectic}}
\label{subsection:main-result-proof}
We will now prove the main theorem. 

\begin{proof} 
By Lemma~\ref{lemma:Pn-superchar-containment}, 
we have that  $J(P(n))\subseteq  J_n$.   
Let us show the reverse inclusion. Suppose by induction that   $J(P(n-2))=J_{n-2}$.

By Theorem~\ref{thm:ds-map-surj}, 
the evaluation map 
\[ 
ds_n: J_n  
\rightarrow J(P(n-2))
\]  
given by $ds_n(f)=f|_{x_{n-1}=x_n^{-1}=t}$ is surjective when restricted to $J(P(n))$. 
Thus, every element of $J_n$  
is a sum of elements from $J(P(n))$ and $\ker {ds}_n$. 
By Proposition~\ref{lemma:kernel-DS-Pn},  $\ker {ds}_n\subseteq J(P(n))$ and   the claim follows. 
\end{proof}

\subsection{The Grothendieck ring of the Lie superalgebra $\mathfrak{p}(n)$}
\label{subsubsection:extend-results-to-pn}
The description of the ring of supercharacters of finite-dimensional representations over the Lie superalgebra $\mathfrak p(n)$ can be deduced from the
description of the ring corresponding to the Lie supergroup $P(n)$.
\begin{proposition}
\label{prop:Pn-vs-pn} 
Denote by $S$ the additive group of $\mathbb C$ and by $T$ its subgroup $\mathbb Z$. 
We have 
\[ 
J(\mathfrak{p}(n))\simeq \mathbb Z[\mathbb C]\otimes_{\mathbb Z[T]} J_n. 
\]  
\end{proposition}
\begin{proof} The character of any simple module can be written in the form  $(x_1\cdots x_n)^a \sch L$, where  $(x_1\cdots x_n)^a$  is a complex power of the
  character of the supertrace representation and $L\in\mathcal F_n$. Such presentation is unique up to the relation $(x_1\cdots x_n)\otimes 1=1\otimes (x_1\cdots x_n)$.
  The statement follows.
  \end{proof}

\subsection{The Grothendieck ring of $SP(n)$ and $\mathfrak{sp}(n)$}
\label{subsubsection:extend-results-to-spn} 

Let $\mathfrak{sp}(n)$ be the special periplectic Lie superalgebra defined as 
\[ 
\mathfrak{sp}(n) = 
\left\{ 
\begin{pmatrix}
A & B \\ 
C & -A^t \\ 
\end{pmatrix} \in \mathfrak{p}(n) : 
\tr(A)=0
\right\},   
\] 
and $SP(n):=P(n)\cap SL(n|n)$ denote the corresponding Lie supergroup. Note that $SP(n)$ has two connected components $\det A=1$ and $\det A=-1$.
The following statement is straightforward.

\begin{proposition}
  The ring $J(SP(n))$ is isomorphic to the quotient ring $J_n/((x_1\cdots x_n-1)(x_1\cdots x_n+1))$ and the ring $J(\mathfrak{sp}(n))$ is isomorphic to
$J_n/(x_1\cdots x_n-1))$.  
\end{proposition}

\subsection{A Weyl groupoid  for $\mathfrak{p}(n)$}
\label{subsection:super-Weyl-groupoid}

In this section, we describe the polynomial invariants of certain affine action of the super Weyl groupoid $\mathfrak{W}$ of $\mathfrak{p}(n)$. 
We follow the definition of the Weyl groupoid given in \cite[Section 9]{MR2776360} (see also  \cite{MR3712194}).

Let $\mathfrak{T}_{\odd}$ be a groupoid with base as the set 
$\{ \pm (\varepsilon_i+\varepsilon_j) :i< j \}
= \pm \Delta(\mathfrak{g}_{-1})$ 
of odd roots. 
The set of morphisms from $\alpha\rightarrow \beta$ is nonempty if and only if $\beta=\pm \alpha$. 
We denote $\tau_{\alpha}$ as the morphism sending $\alpha\mapsto -\alpha$, 
where $\alpha \in \Delta(\mathfrak{g}_{\bar 1})$. 
The group $W=S_n$ acts on $\mathfrak{T}_{\odd}$  by 
$\alpha \mapsto w(\alpha)$ and 
$\tau_{\alpha} \mapsto \tau_{w(\alpha)}$.   
The Weyl groupoid is defined as  
$$
\mathfrak{W}:= W \coprod  W \ltimes \mathfrak{T}_{\odd},
$$ 
where $W$ is considered as a groupoid with a single point base $[W]$ 
and the semi-direct product groupoid 
$W \ltimes \mathfrak{T}_{\odd}$ with the base 
$\Delta(\mathfrak{g}_{-1})$.

Now, define the following affine action $\pi$ of the super Weyl groupoid
$\mathfrak{W}$ on the affine space $V=\mathfrak{h}^*$, which is the dual space to a Cartan subalgebra $\mathfrak{h}$ of $\mathfrak{p}(n)$. 
The base point $[W]$ maps to the space $V$, 
while the base element corresponding to an odd root $\alpha = \varepsilon_i+\varepsilon_j$ 
maps to the hyperplane $\Pi_{\alpha}$ defined by the equation $(\varepsilon_i-\varepsilon_j,x)=0.$
The element $\tau_{\alpha}$ acts as a shift 
\[ 
\tau_{\alpha}(x) = x + \alpha, \qquad \mbox{ where } x \in \Pi_{\alpha}.
\] 
Note that $x + \alpha$ also belongs to $\Pi_{\alpha}$ for every $x\in\Pi_\alpha$. 
We identify $V\cong V^*$ using an invariant bilinear form 
and view the elements of $\mathbb{C}[\mathfrak {h}^*]$ as functions on $V$.
A function $f$ on $V$ is invariant under the action of the groupoid $\mathfrak{W}$
if for any $g \in \mathfrak{W}$, we have $f(g(x)) = f(x)$ for all $x$.

Let  ${P_0} \subseteq \mathfrak{h}^*$ be the  abelian group of the integral  weights of the Lie superalgebra $\mathfrak{p}(n)_0$. The description of $J(\mathfrak p(n))$ can be formulated as follows:

\begin{theorem}
The Grothendieck ring $J(\mathfrak{p}(n))$ of finite-dimensional representations of $\mathfrak{p}(n)$ is isomorphic to the ring $\left(\Span_\mathbb{Z}\{e^\lambda \mid \lambda\in P_0  \}\right)^{\mathfrak{W}}$ of invariants of the super Weyl groupoid $\mathfrak{W}$ under the action described above.  
\end{theorem}
\color{black}
\fi

\newcommand{\etalchar}[1]{$^{#1}$}
\def\cprime{$'$} \def\cprime{$'$} \def\cprime{$'$} \def\cprime{$'$}
\providecommand{\bysame}{\leavevmode\hbox to3em{\hrulefill}\thinspace}
\providecommand{\MR}{\relax\ifhmode\unskip\space\fi MR }
\providecommand{\MRhref}[2]{%
  \href{http://www.ams.org/mathscinet-getitem?mr=#1}{#2}
}
\providecommand{\href}[2]{#2}

\end{document}